\documentclass[12pt]{extarticle}
\usepackage{hyperref}
\usepackage{amsmath, amsthm, amssymb, color}
\usepackage[shortlabels]{enumitem}
\usepackage{graphicx}
\usepackage[all]{xypic}
\usepackage{makecell}
\usepackage{MnSymbol}
\usepackage{dsfont}
\usepackage{comment, float}

\oddsidemargin \evensidemargin
\numberwithin{equation}{section}	
\usepackage{enumitem}
\setenumerate{leftmargin=*}
\allowdisplaybreaks[3]	

\newtheorem{theorem}{Theorem}
\numberwithin{theorem}{section}
\newtheorem{proposition}[theorem]{Proposition}
\newtheorem{lemma}[theorem]{Lemma}

\newtheorem{definition}[theorem]{Definition}

\newtheorem{remark}[theorem]{Remark}
\newtheorem{example}[theorem]{Example}
\newtheorem{conjecture}[theorem]{Conjecture}

\newcommand{\NN}{\mathbb{N}}

\newcommand{\ZZ}{\mathbb{Z}}
\newcommand{\RR}{\mathbb{R}}

\newcommand{\PP}{\mathbb{P}}
\newcommand{\CC}{\mathbb{C}}

\date{}
\title{\textbf{Autocovariance Varieties of \\ Moving Average Random Fields}}
\author{Carlos Am\'endola, Viet Son Pham}

\newcommand{\bbn}{\mathbb{N}}
\newcommand{\bbz}{\mathbb{Z}}
\newcommand{\bbr}{\mathbb{R}}
\newcommand{\bbc}{\mathbb{C}}

\newcommand{\bbp}{\mathbb{P}}

\newcommand{\bone}{\mathds 1}

\newcommand{\calf}{{\cal F}}

\newcommand{\al}{{\alpha}}
\newcommand{\ga}{{\gamma}}
\newcommand{\Ga}{{\Gamma}}

\newcommand{\si}{{\sigma}}

\newcommand{\cov}{{\mathrm{Cov}}}	%Kovarianz
\newcommand{\argmin}{\mathrm{argmin}}
\newcommand{\argmax}{\mathrm{argmax}}

\begin{document}

\maketitle

\begin{abstract}
\noindent We study the autocovariance functions of moving average random fields over the integer lattice $\bbz^d$ from an algebraic perspective. These autocovariances are parametrized polynomially by the moving average coefficients, hence tracing out algebraic varieties. We derive dimension and degree of these varieties and we use their algebraic properties to obtain statistical consequences such as identifiability of model parameters. We connect the problem of parameter estimation to the algebraic invariants known as euclidean distance degree and maximum likelihood degree. Throughout, we illustrate the results with concrete examples. In our computations we use tools from commutative algebra and numerical algebraic geometry.  
% and show the connection to the problem of identifying when a given random field has a moving average representation.
%For the latter, we measure the algebraic complexity by the euclidean distance degree and the maximum likelihood degree.
\end{abstract}

\section{Introduction}

Moving average random fields indexed by the integer lattice $\bbz^d$ generalize the class of discrete-time moving average processes and constitute an important statistical spatial model. They are used to model texture images (cf. \cite{Francos95}), as well as in image segmentation and restoration (cf. \cite{Krishnamurthy96}). Furthermore, they are connected to ARMA (autoregressive moving average) random fields (cf. \cite{Drapatz16} and the references therein) and the sampling problem of CARMA (continuous autoregressive moving average) random fields, in which the autocovariance functions of moving average random fields play a crucial role (cf. \cite[Section~4.3]{Pham18b}). %It is therefore desirable to fully understand these autocovariances and we tackle this problem from an algebraic viewpoint.

A \emph{moving average random field} $(Y_t)_{t\in\bbz^d}$ of order $q=(q_1,q_2, \ldots,q_d) \in \bbn^d$ is defined by the equation
\begin{equation*}
Y_{t}=\sum_{k_1=0}^{q_1}\cdots\sum_{k_d=0}^{q_d} a_k Z_{t-k},\quad t\in\bbz^d,
\end{equation*}
where $k=(k_1, \ldots,k_d)$, $a_k$ are real coefficients and $(Z_t)_{t\in\bbz^d}$ is a real-valued zero-mean white noise (see Definition~\ref{def:MA}).
The autocovariance function 
\begin{equation*}
\ga(t)=\cov[Y_0,Y_t],\quad t\in\bbz^d,
\end{equation*}
for this type of random field is compactly supported, i.e.\ only finitely many values are nonzero. More precisely, we have $\ga(t)=0$ for every $t=(t_1, \ldots,t_d)\in\bbz^d$ with entries satisfying $|t_i|>q_i$ for at least one $i\in\{ 1,\ldots,d \}$. 

We study the autocovariance functions of moving average random fields from an algebraic perspective. Our motivation stems from the field of algebraic statistics \cite{sullivant2018algebraic}. Specifically, inspired by the concept of \textit{moment varieties} \cite{Amendola16}, here we introduce \textit{autocovariance varieties}. The \textit{moving average variety} $\mathcal{MA}_q \subseteq \PP^{N}$ (see Definition~\ref{def:autovar}) is parametrized by  $ (q_1+1) \cdots (q_d +1 )$ moving average coefficients $a_k$ where the indices $k$ satisfy $0\leq k_i \leq q_i$ for $i=1,\ldots,d$. These coefficients induce $(2q_1+1)\cdots (2q_d+1)$ nonzero autocovariance values $\ga(t)$. However, we only consider half of them since the relation $\gamma(-v) = \gamma(v)$ holds for all $v\in \ZZ^d$. 

\begin{example}\label{Ex1}
Let $d=2$ and $q=(1,1)$. Then the $Q=(1+1)(1+1)=4$ parameters $a_{00}, a_{01}, a_{10}, a_{11}$ define the autocovariances
\begin{eqnarray*}
\gamma(0,0) &=& a_{00}^2 + a_{10}^2 + a_{01}^2 + a_{11}^2, \\
\gamma(1,0) &=& a_{00} a_{10} + a_{11} a_{01},  \\
\gamma(0,1) &=& a_{00} a_{01} + a_{11} a_{10}, \\
\gamma(1,-1) &=& a_{10} a_{01},  \\
\gamma(1,1) &=& a_{00} a_{11}. 
\end{eqnarray*}
The moving average variety $\mathcal{MA}_{(1,1)} \subseteq \PP^{4}$ is expected to be 3-dimensional. We characterize it in Theorem \ref{thm:MA11}.
\end{example}

This paper is organized as follows. In Section 2, we give the main definition of a moving average random field and its autocovariance function. We define our main object of study, namely autocovariance varieties, in Section 3. We contrast the properties between moving average processes (one-dimensional) from the higher dimensional moving average random fields. In Theorem~\ref{thm:dimdeg} we establish the dimension and degree of these varieties. In Section 4 we investigate identifiability of the associated models and prove that they are algebraically identifiable. In contrast to the $d=1$ case where the degree of the fiber grows with $q$, we show that for $d>1$ there are generically only two sets of parameters that yield the same autocovariance function. Next, we study two different approaches to estimate model parameters from given samples in Section 5. First, we fit the empirical autocovariance function to the theoretical counterpart using a least squares method. Second, we consider maximum likelihood estimation. Both approaches connect nicely to concepts from algebraic statistics: respectively the ED degree and the ML degree. In Example \ref{ex:simstudy}, we conduct a simulation study comparing classical local optimization methods to numerical homotopy continuation, where we find that the numerical algebraic geometry (NAG) method performs slightly better.

We use the following notation and terminology in this paper: 
%for an integer $k\geq0$ let $[k]$ denotes the set $\{0,...,k\}$. 
The components of a vector $u\in\bbr^d$ are given by $u_1,...,u_d$ if not stated otherwise. If $u,v\in\bbz^d$, then we set $[u,v]:=\{s\in\bbz^d\vert u_i\leq s_i\leq v_i, 1\leq i\leq d\}$, which may be an empty set. The symbol $\preceq$ denotes the lexicographic order and for $x\in\bbr^d$ and $t\in\bbz^d$ we define $x^t:=x_1^{t_1}\cdots x_d^{t_d}$. If $g\colon A\to B$ is a mapping and $y\in g(A)$, then $g^{-1}(y)$ is called fiber of $y$ and each point inside the fiber is called preimage. A statement that holds \textit{generically}, or for a \textit{generic} point, can be interpreted probabilistically as holding for almost all $x\in\bbc^d$ with respect to the  Lebesgue measure. 

\section{Moving Average Random Fields}

Throughout this article, all stochastic objects are defined on a fixed complete probability space $(\Omega,\calf,\bbp)$.
\begin{definition}\label{def:MA}
\begin{enumerate}
\item A random field $(Y_t)_{t\in\bbz^d}$ is called \emph{weakly stationary} if $Y_t\in L^2(\Omega,\calf,\bbp)$ for every $t\in\bbz^d$ and $\ga(t):=\cov[Y_0,Y_t]=\cov[Y_{s},Y_{s+t}]$ for every $t,s\in\bbz^d$. It is called a \emph{white noise} if $\ga(0)>0$ and $\ga(t)=0$ for every $0\neq t\in\bbz^d$. In this case, $\si^2:=\ga(0)$ is called the \emph{white noise variance}.
\item Let $q_1,...,q_d$ be positive integers and $(Z_t)_{t\in\bbz^d}$ be a real-valued zero-mean white noise on $\bbz^d$. A random field $(Y_t)_{t\in\bbz^d}$ is called a \emph{moving average random field} if it satisfies the equation
\begin{equation}\label{MAeq}
Y_{t}=\sum_{k\in[0,q]} a_k Z_{t-k},\quad t\in\bbz^d,
\end{equation}
where $a_k\in\bbr$ such that for each $i=1,...,d$ there exist at least two index vectors $l,m\in[0,q]$ satisfying $a_l\neq0$, $a_m\neq0$, $l_i=q_i$ and $m_i=0$.
\end{enumerate}
\end{definition}
The last condition on the two index vectors $l,m$ guarantees that the $MA(q)$ random field has indeed order $q$ and not a smaller order.
We associate to each $MA(q)$ random field the \emph{moving average polynomial}
\begin{equation*}
\theta(x)=\sum_{k\in[0,q]} a_k x^k,
\end{equation*}
and further, we define the formal backshift operators $B_1,...,B_d$ which act on any random field $(X_t)_{t\in\bbz^d}$ in the following way:
\begin{equation*}
B_i X_t=X_{(t_1,...,t_{i-1},t_i -1,t_{i+1},...,t_d)},\quad t\in\bbz^d,\quad i=1,...,d.
\end{equation*}
With this notation, \eqref{MAeq} can be written in short as
\begin{equation*}
Y_t=\theta(B)Z_t,\quad t\in\bbz^d,
\end{equation*}
where $B=(B_1,...,B_d)$. 

The following proposition establishes the link between the moving average polynomial $\theta$ and the autocovariance function $\ga$ of a $MA(q)$ random field.
\begin{proposition}\label{prop1}
Suppose that $(Y_t)_{t\in\bbz^d}$ is a $MA(q)$ random field driven by a white noise $(Z_t)_{t\in\bbz^d}$ with variance $\si^2$. Then $Y$ is weakly stationary, its autocovariance function $\ga$ is compactly supported and we have
\begin{equation}\label{thetagamma}
\si^2\theta(x)\theta(x^{-1})=\sum_{t\in\bbz^d} \ga(t)x^t.
\end{equation}
\end{proposition}
\begin{proof}
The facts that $Y$ is weakly stationary and $\ga$ is compactly supported are straight-forward.
Let $S=\{ k\in\bbz^d\colon a_k\neq0 \}$ denote the set of indexes with non-vanishing coefficient $a_k$. Then we have that
\begin{align*}
\si^2\theta(x)\theta(x^{-1})&=\si^2\left(\sum_{k\in[0,q]} a_k x^k\right) \left(\sum_{k\in[0,q]} a_k x^{-k}\right)\\
&=\si^2\sum_{t\in\bbz^d} \left( \sum_{k,k+t\in S} a_ka_{k+t} \right) x^t\\
&=\sum_{t\in\bbz^d} \cov\left[\sum_{k\in[0,q]} a_k Z_{0-k},\sum_{k+t\in[0,q]} a_{k+t} Z_{0-k}\right] x^t\\
&=\sum_{t\in\bbz^d} \cov[\theta(B)Z_0,\theta(B)Z_t] x^t\\
&=\sum_{t\in\bbz^d} \cov[Y_0,Y_t] x^t=\sum_{t\in\bbz^d} \ga(t)x^t.
\end{align*}
\end{proof}

\section{Autocovariance Varieties}

We have seen that for $q = (q_1,q_2,\dots,q_d) \in \NN^d$, the autocovariance function of a moving average random field is only dependent on the coefficients $a_k$ of the moving average polynomial $\theta(x)= \sum_{k\in [0,q]} a_k x^k$ and the white noise variance $\si^2$. In order to avoid redundancies in model specification, one can assume without loss of generality that $\si^2=1$ and we will do so for the rest of this paper. There are $Q+1:=\prod_{i=1}^d (q_i+1)$ coefficients $a_k$ and $2N+1:=\prod_{i=1}^d (2q_i+1)$ non-zero autocovariances $\gamma(t)$ for $t \in [-q,q]$. Ordering them in two vectors $a$ and $\ga_a$, we can think of this correspondence as a polynomial map $\Gamma_q :\RR^{Q+1} \mapsto \RR^{2N+1}$ given by $a \mapsto \gamma_{a}$.
	
Since $\gamma(-t)=\gamma(t)$ for all $t\in \ZZ^d$, we can drop half of the autocovariances and only consider $\gamma(t)$ with $t \in [-q,q]$ and $t\succeq0$, where $\succeq$ denotes the lexicographic order. In this way, we have a map $\RR^{Q+1} \mapsto \RR^{N+1}$ which we still denote by $\Gamma_q$. 

The points in the image represent the set of autocovariance functions of moving average random fields. Geometrically, this is a semialgebraic set, defined by polynomial equalities and inequalities. The closure of the image of this parametrization will give a real affine algebraic variety. However, as is standard in algebraic statistics, we will first change the underlying field to be algebraically closed (so the map becomes $\Gamma_q :\CC^{Q+1} \mapsto \CC^{N+1}$ over the complex numbers $\CC$) and then we pass to projective space arriving at $\Gamma_q : \PP^Q \dashedrightarrow  \PP^N$. This last step requires the polynomials $\gamma(t)$ to be homogeneous, and indeed they are in our case.

\begin{definition}\label{def:autovar}
Let $q = (q_1,q_2,\dots,q_d) \in \NN^d$ and define $Q:= \prod_{i=1}^d (q_i+1)-1$ and $N:=  (\prod_{i=1}^d (2q_i+1)-1)/2$. The \textit{autocovariance variety} $\mathcal{MA}_q$ is the image of the autocovariance map 
$\Gamma_q :  \PP^Q \dashedrightarrow \PP^N$. 
\end{definition}

\subsection{Moving average processes}

If $d=1$, moving average random fields are also called moving average processes. These processes are well-studied and belong to the important class of ARMA processes (cf. \cite[Chapter~3]{Brockwell91}). Suppose that $(Y_t)_{t\in\bbz}$ is a $MA(q)$ process given by the equation
\begin{equation*}
Y_{t}=\sum_{k=0}^{q} a_k Z_{t-k},\quad t\in\bbz.
\end{equation*}
Then, the autocovariance function $\ga$ of $Y$ has the simple expression 
\begin{equation*}
\ga(t)=\begin{cases}
\sum_{k=0}^{q-|t|}a_k a_{k+|t|},&\text{if }|t|\leq q,\\
0,&\text{if }|t|> q.
\end{cases}
\end{equation*}

For the class of moving average processes we have that $Q = N = q$. Thus, in this special case the autocovariance map takes the form $$\Gamma_q: \PP^q \dashedrightarrow \PP^q .$$ In the next subsection we will see that the map is actually defined in all of $\PP^q$ (there are no \textit{base points}), so we conclude the following.   
\begin{proposition}\label{prop:d1}
If $d=1$, then $\mathcal{MA}_{q}=\PP^q$.
\end{proposition}

While $\mathcal{MA}_{q}$ is not particularly interesting when $d=1$, the parametrization coming from $\Gamma_q: \PP^q \rightarrow \PP^q$ has interesting fibers and computing them is important for statistical applications. This issue of identifiability will be explored in Section~\ref{sec:vier}.

\begin{remark}
Going back for a moment to the real picture (over $\RR$), the equality $\mathcal{MA}_{q}=\PP^q$ is analogous to the statement that when $d=1$, any autocovariance function with support $[-q,q]$ is an autocovariance function of a $MA(q)$ process \cite[Prop 3.2.1]{Brockwell91}.  
\end{remark}

\subsection{Moving average random fields}

We start by carefully analyzing the case $\mathcal{MA}_{(1,1)}$ mentioned in the introduction.  

\begin{theorem}\label{thm:MA11}
The autocovariance variety $\mathcal{MA}_{(1,1)} \subseteq \PP^{4}$ is a threefold of degree 4. In the polynomial ring with variables $g_t = \gamma(t)$, it is the hypersurface defined by the quartic
\begin{equation*}
g_{10}^2 g_{01}^2
    -g_{00} g_{10} g_{01} g_{11}
    +g_{10}^2 g_{11}^2
    +g_{01}^2 g_{11}^2
    -g_{00} g_{10} g_{01} g_{1-1}
    +g_{00}^2 g_{11} g_{1-1}
    -2 g_{10}^2 g_{11} g_{1-1}
    -2 g_{01}^2 g_{11} g_{1-1} 
\end{equation*}
\begin{equation}\label{eq:quartic}
    -4 g_{11}^3 g_{1-1}
    +g_{10}^2 g_{1-1}^2
     +g_{01}^2 g_{1-1}^2
     +8 g_{11}^2 g_{1-1}^2
     -4 g_{11} g_{1-1}^3 = 0.
\end{equation} Its singular locus is a quadratic surface, which is the union of the three irreducible components corresponding to the prime ideals
\begin{equation}\label{help1}
\left\langle g_{10}-g_{01}, g_{00}-2g_{11}-2g_{1-1}, 4g_{11}g_{1-1}-g_{01}^2 \right\rangle,
\end{equation} 
\begin{equation}\label{help2}
\left\langle g_{10}+g_{01}, g_{00}+2g_{11}+2g_{1-1}, 4g_{11}g_{1-1}-g_{01}^2 \right\rangle,
\end{equation} 
and
\begin{equation}\label{help3}
\left\langle g_{11}-g_{1-1}, g_{00}g_{1-1}-g_{10}g_{01} \right\rangle.
\end{equation} 
\end{theorem}
\begin{proof}
The proof is computational. One way to obtain the quartic \eqref{eq:quartic} is through the following \texttt{Macaulay2} \cite{M2} commands:
\begin{verbatim}
R = QQ[a00,a01,a10,a11]
S = QQ[g00,g01,gm11,g10,g11]
h = map(R,S,{ a00^2  + a10^2  + a01^2  + a11^2,    a00*a01 + a10*a11,
            a10*a01, a00*a10 + a01*a11,  a00*a11} ) 
I = kernel h
\end{verbatim}
For the singular locus, we compute the radical ideal of the quartic along with its vanishing gradient, and then compute its prime decomposition.
\end{proof}

\begin{remark}
Substituting the parametrization of Example~\ref{Ex1} into \eqref{help1} to \eqref{help3}, we find that the three irreducible components of the singular locus correspond to the three conditions
\begin{equation}\label{help4}
a_{10}=a_{01}\, \text{ and } \,a_{00}=a_{11},
\end{equation}
\begin{equation}\label{help5}
a_{10}=-a_{01}\, \text{ and } \,a_{00}=-a_{11},
\end{equation}
and\begin{equation}\label{help6}
a_{00}a_{11}=a_{01}a_{10}.
\end{equation}
These conditions represent submodels and we will analyze Equation~\eqref{help6} in more detail in Example~\ref{ex:ma(1,1)ident}.
\end{remark}

The complexity of $\mathcal{MA}_q$ increases rapidly when $d>1$. It is computationally challenging to obtain generators for its prime ideal even for small values of $q$ and $d$. Beyond $q=(1,1)$, we were also able to do this for $q=(1,2)$ and $q=(1,1,1)$.

\begin{proposition}
The autocovariance variety $\mathcal{MA}_{(1,2)} \subseteq \PP^{7}$ is 5-dimensional of degree 16. Its prime ideal is cut out by 7 sextics. One of those is
\begin{small}
\begin{multline*}
	4g_{12}^3g_{1-2}g_{11}^2-28g_{12}^2g_{1-2}^2g_{11}^2+28g_{12}g_{1-2}^3g_{11}^2-4
	g_{1-2}^4g_{11}^2-4g_{1-2}^2g_{11}^4-g_{12}^2g_{11}^2g_{02}^2+6g_{12}g_{1-2}g_{11}^2g_{02}^2\\
	-5g_{1-2}^2g_{11}^2g_{02}^2-32g_{12}^3g_{1-2}g_{11}g_{1-1}+64g_{12}^2g_{1-2}^2g_{11}g_{1-1}-32g_{12}
	g_{1-2}^3g_{11}g_{1-1}\\
	-8g_{12}g_{1-2}g_{11}^3g_{1-1}+8g_{1-2}^2g_{11}^3g_{1-1}+6g_{12}^2g_{11}g_{02}^
	2g_{1-1}-12g_{12}g_{1-2}g_{11}g_{02}^2g_{1-1}+6g_{1-2}^2g_{11}g_{02}^2g_{1-1}-4g_{12}^4g_{1-1}^2\\
	+28g_{12}^3g_{1-2}g_{1-1}^2-28g_{12}^2g_{1-2}^2g_{1-1}^2+4g_{12}g_{1-2}^3g_{1-1}^2-4g_{12}^2
	g_{11}^2g_{1-1}^2+16g_{12}g_{1-2}g_{11}^2g_{1-1}^2-4g_{1-2}^2g_{11}^2g_{1-1}^2
	\\
	-5g_{12}^2g_{02}^2
	g_{1-1}^2+6g_{12}g_{1-2}g_{02}^2g_{1-1}^2-g_{1-2}^2g_{02}^2g_{1-1}^2+8g_{12}^2g_{11}g_{1-1}^3-8
	g_{12}g_{1-2}g_{11}g_{1-1}^3-4g_{12}^2g_{1-1}^4+8g_{12}^2g_{1-2}g_{11}^2g_{10}
	\\
	-16g_{12}g_{1-2}^2
	g_{11}^2g_{10}+8g_{1-2}^3g_{11}^2g_{10}-2g_{12}g_{11}^2g_{02}^2g_{10}+6g_{1-2}g_{11}^2g_{02}^2g_{10}+
	4g_{12}g_{11}g_{02}^2g_{1-1}g_{10}\\
	+4g_{1-2}g_{11}g_{02}^2g_{1-1}g_{10}+8g_{12}^3g_{1-1}^2g_{10}-16
	g_{12}^2g_{1-2}g_{1-1}^2g_{10}+8g_{12}g_{1-2}^2g_{1-1}^2g_{10}+6g_{12}g_{02}^2g_{1-1}^2g_{10}\\
	-2g_{1-2}g_{02}^2g_{1-1}^2g_{10}
	+4g_{12}g_{1-2}g_{11}^2g_{10}^2-4g_{1-2}^2g_{11}^2g_{10}^2-g_{11}^2
	g_{02}^2g_{10}^2-2g_{11}g_{02}^2g_{1-1}g_{10}^2-4g_{12}^2g_{1-1}^2g_{10}^2\\
	+4g_{12}g_{1-2}g_{1-1}^2
	g_{10}^2-g_{02}^2g_{1-1}^2g_{10}^2+12g_{12}^2g_{1-2}g_{11}g_{02}g_{01}-16g_{12}g_{1-2}^2g_{11}g_{02}
	g_{01}+4g_{1-2}^3g_{11}g_{02}g_{01}\\
	+4g_{1-2}g_{11}^3g_{02}g_{01}-2g_{12}g_{11}g_{02}^3g_{01}+2g_{1-2}
	g_{11}g_{02}^3g_{01}+4g_{12}^3g_{02}g_{1-1}g_{01}\\
	-16g_{12}^2g_{1-2}g_{02}g_{1-1}g_{01}+12g_{12}g_{1-2}^
	2g_{02}g_{1-1}g_{01}+4g_{12}g_{11}^2g_{02}g_{1-1}g_{01}-8g_{1-2}g_{11}^2g_{02}g_{1-1}g_{01}\\
	+2g_{12}
	g_{02}^3g_{1-1}g_{01}-2g_{1-2}g_{02}^3g_{1-1}g_{01}-8g_{12}g_{11}g_{02}g_{1-1}^2g_{01}+4g_{1-2}g_{11}
	g_{02}g_{1-1}^2g_{01}+4g_{12}g_{02}g_{1-1}^3g_{01}\\
	-8g_{1-2}^2g_{11}g_{02}g_{10}g_{01}-2g_{11}g_{02}^3
	g_{10}g_{01}-8g_{12}^2g_{02}g_{1-1}g_{10}g_{01}-2g_{02}^3g_{1-1}g_{10}g_{01}+4g_{1-2}g_{11}g_{02}g_{10}^2
	g_{01}\\
	+4g_{12}g_{02}g_{1-1}g_{10}^2g_{01}-g_{12}^2g_{02}^2g_{01}^2+2g_{12}g_{1-2}g_{02}^2g_{01}^2-
	g_{1-2}^2g_{02}^2g_{01}^2-g_{11}^2g_{02}^2g_{01}^2+2g_{11}g_{02}^2g_{1-1}g_{01}^2-g_{02}^2g_{1-1}^2
	g_{01}^2\\
	+2g_{12}g_{02}^2g_{10}g_{01}^2+2g_{1-2}g_{02}^2g_{10}g_{01}^2-g_{02}^2g_{10}^2g_{01}^2-4
	g_{1-2}^2g_{11}^2g_{02}g_{00}-8g_{12}g_{1-2}g_{11}g_{02}g_{1-1}g_{00}-4g_{12}^2g_{02}g_{1-1}^2g_{00}\\
	-4
	g_{12}^2g_{1-2}g_{11}g_{01}g_{00}+12g_{12}g_{1-2}^2g_{11}g_{01}g_{00}+g_{12}g_{11}g_{02}^2g_{01}g_{00}+
	g_{1-2}g_{11}g_{02}^2g_{01}g_{00}+12g_{12}^2g_{1-2}g_{1-1}g_{01}g_{00}\\
	-4g_{12}g_{1-2}^2g_{1-1}g_{01}
	g_{00}+g_{12}g_{02}^2g_{1-1}g_{01}g_{00}+g_{1-2}g_{02}^2g_{1-1}g_{01}g_{00}-4g_{12}g_{1-2}g_{11}g_{10}g_{01}
	g_{00}+g_{11}g_{02}^2g_{10}g_{01}g_{00}\\
	-4g_{12}g_{1-2}g_{1-1}g_{10}g_{01}g_{00}+g_{02}^2g_{1-1}g_{10}g_{01}
	g_{00}-4g_{12}g_{1-2}g_{02}g_{01}^2g_{00}+g_{1-2}^2g_{11}^2g_{00}^2+2g_{12}g_{1-2}g_{11}g_{1-1}g_{00}^2\\
	+
	g_{12}^2g_{1-1}^2g_{00}^2-g_{1-2}g_{11}g_{02}g_{01}g_{00}^2-g_{12}g_{02}g_{1-1}g_{01}g_{00}^2+g_{12}g_{1-2}
	g_{01}^2g_{00}^2
\end{multline*}
\end{small}
The autocovariance variety $\mathcal{MA}_{(1,1,1)} \subseteq \PP^{13}$ is 7-dimensional of degree 64. Its prime ideal is cut out by 56 quartics, 90 quintics and 50 sextics.
\end{proposition}

Table \ref{table1} presents the basic properties of the first autocovariance varieties $\mathcal{MA}_{(q_1,q_2)}$, that is, for $d=2$. The dimension appears to be the expected one, while the degree follows a clear pattern as a power of two. We will prove that this actually holds for any $\mathcal{MA}_q$. To that end, we use the next two lemmas. 
\begin{table}[ht]
\centering
\begin{tabular}{cccccc}
\hline
$q_1$ & $q_2$ & $\dim (\mathcal{MA}_q)$ & $N$  & $\deg (\mathcal{MA}_q)$ & generators \\ 
\hline
1 & 1 & 3 & 4 & 4 & 1 quartic \\ 
1 & 2 & 5 & 7 & 16 & 7 sextics \\ 
1 & 3 & 7 & 10 & 64 & ? \\
2 & 2 & 8 & 12 & 128 & ? \\
\hline
\end{tabular}
\caption{Summary of first autocovariance varieties for $d=2$}\label{table1}
\end{table}

\begin{lemma}\label{lem:basepoints}
The map $\Gamma_q: \PP^Q \mathrel{-\,}\dashedrightarrow \PP^N$ has no base points.
\end{lemma}
\begin{proof}
Assume that $\Gamma_q(a)=0$. We know from \eqref{thetagamma} that $$\theta(x)\theta(x^{-1})=\sum_{t\in\bbz^d} \ga(t)x^t = 0. $$ Multiplying both sides by the monomial $x^q = x_1^{q_1}\cdots x_d^{q_d}$, we obtain the product of two polynomials $\theta(x) \cdot x^q\theta(x^{-1})$ that equals the zero polynomial. Since the polynomial ring $K[x]$ is an integral domain when $K$ is a field, we must have that either $\theta(x)=0$ or $x^q\theta(x^{-1})=0$. In particular, all the coefficients $a_k=0$, that is, $a=0$ is the zero vector.
\end{proof}

\begin{lemma}\label{lem:veronese}
The autocovariance variety $\mathcal{MA}_q$ is a linear projection of the Veronese variety.
Furthermore, $\ga(t)$ is the sum of exactly $(q_1-|t_1|+1)\cdots(q_d-|t_d|+1)$ quadratic monomials for every $t\in[-q,q]$ with $t\succeq 0$, and each monomial appears exactly once. 
\end{lemma}
\begin{proof}
The quadratic Veronese embedding precisely consists of all quadratic monomials.
The parametrization of $\mathcal{MA}_q$ consists of quadrics, each one is a sum of quadratic monomials.
Moreover, Proposition~\ref{prop1} implies that
\begin{equation}
\ga(t)=\sum_{k,k+t\in [0,q]}a_k a_{k+t},
\end{equation} 
for every $t\in[-q,q]$, which shows the second part of the assertions.
\end{proof}

Now we state the main theorem concerning our varieties $\mathcal{MA}_q$.
\begin{theorem}\label{thm:dimdeg}
Let $q\in\bbn^d$. Then
$$\dim (\mathcal{MA}_q) = Q = \prod_{i=1}^d (q_i+1)-1 $$
and if $d>1$, then 
$$\deg (\mathcal{MA}_q)=2^{Q-1} = 2^{\prod_{i=1}^d (q_i+1)-2}.$$
%\begin{align*}
%\dim (\mathcal{MA}_q) = Q = \prod_{i=1}^d (q_i+1)-1\\
%amb=N=\frac{\prod_{i=1}^d (2q_i+1)-1}{2}\\
%2N=\prod_{i=1}^d (2q_i+1)-1\\
%\deg (\mathcal{MA}_q)=2^{Q-2} = 2^{\prod_{i=1}^d (q_i+1)-2}
%\end{align*}
\end{theorem}
\begin{proof}
Let $D:=\dim(\mathcal{MA}_q)$ denote the dimension of $\mathcal{MA}_q$ and consider the regular map $\Gamma_q: \PP^Q \longrightarrow \PP^N$. Since the domain is $Q$-dimensional, the inequality $D \leq Q$ has to hold. However, $\Gamma_q$ is not a constant map and has no base points by Lemma~\ref{lem:basepoints}. As there are no nonconstant regular maps from a projective space to a variety of smaller dimension, we must have equality, that is, $D = Q$. 

Now consider the degree of the map $\Gamma_q$:
 $$\deg(\Gamma_q) = \deg(\mathcal{MA}_q) \deg(\Gamma^{-1}(\gamma))$$
where $\gamma \in\mathcal{MA}_q $ is generic. By Lemmas \ref{lem:basepoints} and \ref{lem:veronese}, the left hand side equals the degree of the quadratic Veronese variety $\mathcal{V}_{Q,2}$, which is $2^Q$. In addition, the identifiability Theorem~\ref{thm:ident} below proves that $\deg(\Gamma^{-1}(\gamma)) = 2$. Hence, we conclude that $\deg(\mathcal{MA}_q) = 2^{Q-1}$.
\end{proof}

\section{Identifiability} \label{sec:vier}

We show that all models $MA(q)$ with $q \in \NN^d$ are algebraically identifiable (in the sense of \cite{Amendola18}). This means that the map from the model parameters to the autocovariances is generically finite to one. 

\subsection{Moving Average Processes}

The following result is the projective version of the known result in the moving average process literature \cite{Brockwell91}.

\begin{proposition}\label{prop:ident1}
If $d=1$, the fibers of a generic point $\gamma\in \mathcal{MA}_{q}$ consist of $2^q$ points. 
\end{proposition}
\begin{proof}
Let $\al_1,...,\al_q\in\bbc$ be the $q$ roots of the moving average polynomial
\begin{equation*}
\theta(x)=\sum_{k=0}^{q} a_k x^k=a_q(x-\al_1)\cdots(x-\al_q),\quad x\in\bbc.
\end{equation*}
Using Proposition~\ref{prop1}, we see that there are exactly $2^{q+1}$ polynomials which generate $\ga$ as above, all of which have the form
\begin{equation*}
\pm a_q(x-\al_1)_{\pm}\cdots(x-\al_q)_{\pm},
\end{equation*}
where
\begin{equation}\label{flip}
(x-\al_i)_{+}:=(x-\al_i)\quad\text{and}\quad(x-\al_i)_{-}:=(\al_i x-1).
\end{equation}\label{eq:rev}
Hence, the fiber of any point in $\mathcal{MA}_{q}$ under $\Gamma_q: \PP^q \rightarrow \PP^q$ consists in general of $2^{q}$ points.
\end{proof}

Proposition~\ref{prop:ident1} has two consequences. First, it implies that the map $\Gamma_q$ is not injective and the moving average parameters $a_i$ are not identifiable from a second order point of view if $d=1$. Second, it is possible to deduce all preimage points from a single one by inverting the roots of $\theta$ as suggested in \eqref{flip}.

In order to obtain injectivity of $\Gamma_q$, one usually imposes the condition that all roots $\alpha_i$ of the polynomial $\theta$ lie strictly outside the unit disk (and $a_0>0$). This property is also called \emph{invertibility} since it holds if and only if there exists coefficients $\pi_0,\pi_1,...$ with $\sum_{k=0}^\infty |\pi_k|<\infty$ such that the white noise sequence $(Z_t)_{t\in\bbz}$ can be expressed as
\begin{equation*}
Z_{t}=\sum_{k=0}^{\infty} \pi_k Y_{t-k},\quad t\in\bbz.
\end{equation*}

\begin{example}
Let $q=1$. This is the simplest moving average model $MA(1)$. We have that $\theta(x)= a_0 + a_1 x$ and $\Gamma_1: \PP^1 \longrightarrow \PP^1 $ is given by $$\Gamma_1(a_0,a_1) = (a_0^2 + a_1^2 , a_0a_1). $$
The fiber of a generic point $\gamma = (\gamma_0,\gamma_1)$ consists of $2=2^1$ points in $\PP^1$. They are $(\tilde{a}_0, \tilde{a}_1)$ and $(\tilde{a}_1, \tilde{a}_0)$ where
\begin{equation}
\tilde{a}_0 = \sqrt{\frac{ \ga_{0}+\sqrt{ \ga_{0}^2-4  \ga_{1}^2}}{2}},\qquad
\tilde{a}_1 = \sqrt{\frac{ \ga_{0}-\sqrt{ \ga_{0}^2-4  \ga_{1}^2}}{2}}.    
\end{equation}
The invertibility condition is equivalent to $|a_0|>|a_1|$. 
\end{example}

The observed symmetry of the two points $(\tilde{a}_0, \tilde{a}_1)$ and $(\tilde{a}_1, \tilde{a}_0)$ above extends to higher $q$. In fact, it holds that
\begin{equation}\label{sym}
\Gamma_q(a_0,a_1,\dots,a_{q-1},a_q) = \Gamma_q(a_q,a_{q-1},\dots,a_1,a_0).
\end{equation}
This can be seen from \eqref{thetagamma}, where the reversal occurs by inverting all the roots in \eqref{eq:rev}.  

For general $q>1$, there exist algorithms to numerically approximate the invertible solution with $a_0=1$. A basic one is the \emph{ innovations algorithm}, which recursively converges to the moving average parameters $a_k$ given the autocovariance values $\ga(t)$ under the invertibility condition (we refer to Section~2 in \cite{Brockwell88} for details). Other approaches use spectral factorization methods  \cite{sayed2001survey}. While we do not pursue this in this paper, the fact remains that the desired parameters are solutions to a polynomial system of equations, so it would be interesting to compare these with state-of-the-art algorithms in numerical algebraic geometry. See Example \ref{ex:simstudy} for an illustration of such techniques.

Furthermore, the symmetry in the polynomial system means that one does not necessarily need to find a root of a polynomial of degree $2^q$ even when there are $2^q$ solutions. We illustrate this with $q=2$.

\begin{example}
For $q=2$ we have $\theta(x)= a_0 + a_1 x + a_2 x^2$ and $\Gamma_2: \PP^2 \longrightarrow \PP^2 $ given by $$\Gamma_2(a_0,a_1,a_2) = (a_0^2 + a_1^2 + a_2^2, a_0a_1+ a_1a_2, a_0 a_2). $$
The fiber of a generic point $\gamma = (\gamma_0,\gamma_1,\ga_2)\in\mathcal{MA}_{2}$ consists of $2^2=4$ points. A Gr\" obner basis elimination from the system $\Gamma_2(a_0,a_1,a_2)= (\gamma_0, \gamma_1, \gamma_2)$  with order $a_2,a_0,a_1$  reveals a triangular system with a quadric in $a_1^2$:
\begin{align*}
a_1^4 - (\ga_0 + 2\ga_2) a_1^2 + \ga_1^2 &= 0 \\
a_0^2 a_1 - \ga_1a_0 + a_1\ga_2&= 0 \\
a_2 a_1 + a_0 a_1 - \ga_1 &= 0.
\end{align*}
And hence the solutions for $(a_0,a_1,a_2)$ in terms of $(\gamma_0,\gamma_1,\gamma_2)$ can be obtained as
$$
a_1= \sqrt{\frac{\ga_0+2\ga_2\pm\sqrt{(\ga_0+2\ga_2)^2-4\ga_1^2}}{2}} , \qquad
a_0= \sqrt{\frac{\ga_1 \pm \sqrt{\ga_1^2 - 4a_1^2 \ga_2}}{2a_1}} , \qquad
a_2= \frac{\ga_1 - a_0 a_1}{a_1}.
$$
\begin{comment}
\begin{align*}
a_1&= \sqrt{\frac{g_0+2g_2\pm\sqrt{(g_0+2g_2)^2-4g_1^2}}{2}} \\
a_0&= \sqrt{\frac{g_1 \pm \sqrt{g_1^2 - 4a_1^2 g_2}}{2a_1}}\\
a_2&= \frac{g_1 - a_0 a_1}{a_1}.
\end{align*}
\end{comment}

\begin{comment}
\begin{tiny}
\begin{align*}
a_0&=-\frac{1}{2} \sqrt{-\sqrt{g_0^2+4 g_0 g_2-4 g_1^2+4 g_2^2}-\sqrt{2} \sqrt{-g_0 \sqrt{g_0^2+4 g_0 g_2-4 g_1^2+4 g_2^2}+2
   g_2 \sqrt{g_0^2+4 g_0 g_2-4 g_1^2+4 g_2^2}+g_0^2-2 g_1^2-4 g_2^2}+g_0-2 g_2}\\
a_1&=-\frac{\sqrt{-\sqrt{g_0^2+4 g_0 g_2-4 g_1^2+4 g_2^2}-\sqrt{2} \sqrt{-g_0 \sqrt{g_0^2+4 g_0 g_2-4 g_1^2+4 g_2^2}+2
   g_2 \sqrt{g_0^2+4 g_0 g_2-4 g_1^2+4 g_2^2}+g_0^2-2 g_1^2-4 g_2^2}+g_0-2 g_2} \left(\sqrt{2} \left(\sqrt{g_0^2+4
   g_0 g_2-4 g_1^2+4 g_2^2}+g_0+2 g_2\right) \sqrt{-g_0 \sqrt{g_0^2+4 g_0 g_2-4 g_1^2+4 g_2^2}+2 g_2
   \sqrt{g_0^2+4 g_0 g_2-4 g_1^2+4 g_2^2}+g_0^2-2 g_1^2-4 g_2^2}+4 g_1^2\right)}{16 g_1 g_2}\\
a_2&=\frac{\sqrt{-\sqrt{g_0^2+4 g_0 g_2-4 g_1^2+4 g_2^2}-\sqrt{2} \sqrt{-g_0 \sqrt{g_0^2+4 g_0 g_2-4 g_1^2+4 g_2^2}+2
   g_2 \sqrt{g_0^2+4 g_0 g_2-4 g_1^2+4 g_2^2}+g_0^2-2 g_1^2-4 g_2^2}+g_0-2 g_2} \left(\sqrt{g_0^2+4 g_0
   g_2-4 g_1^2+4 g_2^2}-\sqrt{2} \sqrt{-g_0 \sqrt{g_0^2+4 g_0 g_2-4 g_1^2+4 g_2^2}+2 g_2 \sqrt{g_0^2+4 g_0
   g_2-4 g_1^2+4 g_2^2}+g_0^2-2 g_1^2-4 g_2^2}-g_0+2 g_2\right)}{8 g_2}
\end{align*}
\end{tiny}

\end{comment}
\end{example}

\subsection{Moving Average Random Fields}

The following result demonstrates a fundamental difference between $d=1$ and $d>1$ in terms of identifiability. On the other hand, it shows how the symmetry in \eqref{sym} generalizes to higher dimensions.

\begin{theorem}\label{thm:ident}
Suppose that the moving average polynomial $\theta$ is generic. Then for $d>1$, the fibers of a point $\gamma \in \mathcal{MA}_{q}$ are only two points $a$ and $a'$ in $\PP^Q$. One is obtained from the other by $a'_k = a_{q-k}$ for any $k \in [0,q]$.
\end{theorem}
\begin{proof}
Let $\ga$ be the image of the coefficients $a_k$ of a moving average polynomial $\theta$ under the mapping $\Ga_q$ and assume that $\theta'$ is another polynomial which also generates $\ga$ and has coefficients $a'_k$. Due to Proposition~\ref{prop1}, the polynomial equation
\begin{equation}
\theta(x)  (x^q\theta(x^{-1}))=\theta'(x)  (x^q\theta'(x^{-1}))
\end{equation}
has to hold. Since generically $\theta$ is irreducible, we either have $\theta'=\theta$ or $\theta'=x^q\theta(x^{-1})$, which proves the assertion.
\end{proof}

\begin{example}\label{ex:ma(1,1)ident}
We consider again the autocovariance variety $\mathcal{MA}_{(1,1)}$ and assume that $\ga\in\mathcal{MA}_{(1,1)}$ is generated by a generic moving average polynomial $\theta$ as in the setting of Theorem~\ref{thm:ident}, that is, $\theta$ is irreducible. Then the fiber of $\gamma$ is given by the equations
\begin{align*}
a_{00}&=\sqrt{\frac{\ga_{00} \ga_{11}-\ga_{01}\ga_{10}-\sqrt{(\ga_{01} \ga_{10}-\ga_{00} \ga_{11})^2-4 \ga_{11}^2 (\ga_{11}-\ga_{1-1})^2}}{2(\ga_{11}-\ga_{1-1})}},\\
a_{10}&=\sqrt{\frac{-\ga_{00} \ga_{1-1}+\ga_{01}\ga_{10}-\sqrt{(\ga_{01} \ga_{10}-\ga_{00} \ga_{1-1})^2-4 \ga_{1-1}^2 (\ga_{11}-\ga_{1-1})^2}}{2(\ga_{11}-\ga_{1-1})}},\\
a_{01}&=\sqrt{\frac{-\ga_{00} \ga_{1-1}+\ga_{01}\ga_{10}+\sqrt{(\ga_{01} \ga_{10}-\ga_{00} \ga_{1-1})^2-4 \ga_{1-1}^2 (\ga_{11}-\ga_{1-1})^2}}{2(\ga_{11}-\ga_{1-1})}},\\
a_{11}&=\sqrt{\frac{\ga_{00} \ga_{11}-\ga_{01}\ga_{10}+\sqrt{(\ga_{01} \ga_{10}-\ga_{00} \ga_{11})^2-4 \ga_{11}^2 (\ga_{11}-\ga_{1-1})^2}}{2(\ga_{11}-\ga_{1-1})}},
\end{align*}
and $a'_{00}=a_{11}$, $a'_{01}=a_{10}$, $a'_{10}=a_{01}$, $a'_{11}=a_{00}$.
Substituting in the formulas from Example~\ref{Ex1}, we observe that the discriminants 
\begin{align*}
(\ga_{01} \ga_{10}-\ga_{00} \ga_{11})^2-4 \ga_{11}^2 (\ga_{11}-\ga_{1-1})^2,\\
(\ga_{01} \ga_{10}-\ga_{00} \ga_{1-1})^2-4 \ga_{1-1}^2 (\ga_{11}-\ga_{1-1})^2,
\end{align*}
are equal to
\begin{align*}
\left(a_{00}^2-a_{11}^2\right)^2 (a_{01} a_{10}-a_{00} a_{11})^2,\\
\left(a_{01}^2-a_{10}^2\right)^2 (a_{01} a_{10}-a_{00} a_{11})^2,
\end{align*}
which are nonnegative in the real case (as they should be when the moving average parameters are real).

If however $\theta$ is not irreducible, it has to be the product of two linear factors. Then identifiability from the above theorem fails and we have up to 4 preimages in the fiber of $\gamma$. We note that this explains the irreducible component \eqref{help3} of the singular locus from Theorem \ref{thm:MA11}, which is equivalent to \eqref{help6}. 
In order to see this, we first assume that Equation~\eqref{help6} holds. This implies that there exists a constant $b^*\in\bbr$ such that we have $a_{01}=b^*a_{00}$ and $a_{11}=b^*a_{10}$. Thus, the polynomial $\theta$ satisfies
\begin{equation*}
\theta(x)=(a_{00}+a_{10}x_1)(1+b^*x_2)
\end{equation*}
and is therefore reducible.
On the other hand, assuming that 
\begin{equation*}
\theta(x)=a_{00}+a_{10}x_1+a_{01}x_2+a_{11}x_1x_2=(b_{10}+b_{11}x_1)(b_{20}+b_{21}x_2)
\end{equation*}
for some real-valued coefficients $b_{10},b_{11},b_{20},b_{21}$, we can deduce \eqref{help6}. One of the four preimage points in the fiber of $\ga$ is given by the equations
\begin{equation*}
a_{00}=\frac{1}{2} \sqrt{\ga_{00}+\sqrt{\ga_{00}^2-4 \left(\ga_{01}^2+\ga_{10}^2-4 \ga_{11}^2\right)}- \sqrt{2\ga_{00} \sqrt{\ga_{00}^2-4 \left(\ga_{01}^2+\ga_{10}^2-4
   \ga_{11}^2\right)}+2\ga_{00}^2-4 \left(\ga_{01}^2+\ga_{10}^2\right)}},
\end{equation*}
\begin{equation*}
a_{01}=\frac{1}{2} \sqrt{\ga_{00}-\sqrt{\ga_{00}^2-4 \left(\ga_{01}^2+\ga_{10}^2-4 \ga_{11}^2\right)}-\sqrt{-2 \ga_{00} \sqrt{\ga_{00}^2-4 \left(\ga_{01}^2+\ga_{10}^2-4
   \ga_{11}^2\right)}+2 \ga_{00}^2-4 \left(\ga_{01}^2+\ga_{10}^2\right)}},
\end{equation*}
\begin{equation*}
a_{10}=\frac{1}{2} \sqrt{\ga_{00}-\sqrt{\ga_{00}^2-4 \left(\ga_{01}^2+\ga_{10}^2-4 \ga_{11}^2\right)}+\sqrt{-2 \ga_{00} \sqrt{\ga_{00}^2-4 \left(\ga_{01}^2+\ga_{10}^2-4
   \ga_{11}^2\right)}+2 \ga_{00}^2-4 \left(\ga_{01}^2+\ga_{10}^2\right)}},
\end{equation*}
\begin{equation*}
a_{11}=\frac{1}{2} \sqrt{\ga_{00}+\sqrt{\ga_{00}^2-4 \left(\ga_{01}^2+\ga_{10}^2-4 \ga_{11}^2\right)}+ \sqrt{2\ga_{00} \sqrt{\ga_{00}^2-4 \left(\ga_{01}^2+\ga_{10}^2-4
   \ga_{11}^2\right)}+2\ga_{00}^2-4 \left(\ga_{01}^2+\ga_{10}^2\right)}}.
\end{equation*}
\end{example}
\begin{remark}
If, in contrast to the setting of Theorem~\ref{thm:ident}, $\theta$ is not irreducible, then the fiber of $\ga$ under $\Ga_q$ consists of more than two preimages, as illustrated in the previous example. By \eqref{eq:rev}, the maximum number of preimages is $2^{q_1+\cdots+q_d}$ and occurs exactly when $\theta$ is completely separable, that is, a product of linear forms.
\end{remark}

\section{Parameter Estimation}

In this section we go one step further and consider the problem of parameter estimation from observed sample points. We consider two methods: least squares estimation and maximum likelihood estimation. Both involve solving polynomial systems of equations. Algebraically, the computational complexity of the estimation problem is measured by the \textit{ED degree} \cite{draisma2016euclidean} of the associated variety in the first case and by the \textit{ML degree} \cite{Amendola19, catanese2006maximum} in the second.

\subsection{Least squares estimation}

Let  $(Y_t)_{t\in\bbz^d}$ be a $MA(q)$ random field, which by Definition~\ref{def:MA} has mean zero. If we are given observations of $Y$ on a lattice $L=\{ 1,...,n \}^d$, we can estimate the autocovariance function $\ga(t)$ by the empirical autocovariance estimator
\begin{equation*}
\hat{\ga}_n(t):=\frac{1}{|B_{n,t}|}\sum_{s\in B_{n,t}}Y(t+s)Y(s), \quad t \in [-q,q],\quad t\succeq0,
\end{equation*}
where
\begin{equation*}
B_{n,t}:=\{ s\in\bbz^d \vert s,s+t\in L \} \text{ and }|B_{n,t}|=\prod_{i=1}^d (n-|t_i|)\bone_{\{|t_i|\leq n\}}.
\end{equation*}
If $\hat{\ga}_n(t)$ were exact values, we would be in the situation of the previous section. However, these are just numerical estimates which form a point $\hat{\ga}_n$ that almost surely lies outside the model $\mathcal{MA}_q$. 
One approach is to project the estimated vector $\hat{\ga}_n$ onto the autocovariance variety $\mathcal{MA}_{q}$, that is, obtaining $\ga_n^{*}\in\mathcal{MA}_{q}$ which has the smallest Euclidean distance to $\hat{\ga}_n$:
\begin{equation}\label{eq:lse}
\ga_n^{*}:=\argmin_{\ga\in\mathcal{MA}_{q}}\|\ga-\hat{\ga}_n\|.
\end{equation}
The number of critical points of this least squares optimization problem is counted by the Euclidean distance degree (\emph{ED degree}). 
\begin{proposition}\label{prop:EDMA11}
The ED degree of $\mathcal{MA}_q$ is 1 if $d=1$. The ED degree of $\mathcal{MA}_{(1,1)}$ is 16.
\end{proposition}
\begin{proof}
The first part is a consequence of Proposition \ref{prop:d1}. In fact, for $d=1$ the unique critical point for \eqref{eq:lse} is $\ga_n^{*} = \hat{\ga}_n$. For the second we use the following \texttt{M2} code:
\begin{verbatim}
R = QQ[g00,g01,gm11,g10,g11]
I = ideal(g01^2*gm11^2-g00*g01*gm11*g10+g01^2*g10^2+gm11^2*g10^2+
    g00^2*gm11*g11-2*g01^2*gm11*g11-4*gm11^3*g11-g00*g01*g10*g11
    -2*gm11*g10^2*g11+g01^2*g11^2+8*gm11^2*g11^2+g10^2*g11^2-4*gm11*g11^3) 
sing = ideal singularLocus I
u = {5,7,13,11,3};
M = (matrix{apply(# gens R,i->(gens R)_i-u_i)})||(transpose(jacobian I));
time J = saturate(I + minors(2,M), sing);
dim J, degree J
\end{verbatim}
The vector $u$ represents a generic choice of $\ga$ and the saturation is needed to remove the critical points that lie in the singular locus.
\end{proof}

 We illustrate with an example: 

\begin{example}\label{ex:MA11}
We simulate $2500$ points of a $MA(1,1)$ random field on a $50\times50$ grid in \texttt{R} (see Figure~\ref{fig1}). As white noise we take an i.i.d. standard Gaussian random field. The moving average parameters are chosen as
\begin{equation*}
a_{00}=7,\qquad a_{01}=-5,\qquad a_{10}=3,\qquad a_{11}=1
\end{equation*}
and the corresponding autocovariances values are
\begin{equation*}
\ga = (\ga_{00},\ga_{01},\ga_{10},\ga_{11},\ga_{1-1}) = (84, 16, -32 , 7 , -15).
\end{equation*}
After centering the sample, we compute the empirical autocovariances
\begin{equation*}
\hat{\ga}_n=(\hat{\ga}_n(0,0),\hat{\ga}_n(0,1),\hat{\ga}_n(1,0),\hat{\ga}_n(1,1),\hat{\ga}_n(1,-1)).
\end{equation*}
By Proposition \ref{prop:EDMA11}, we expect 16 complex critical points, and we compute them numerically.
%with the \texttt{Julia} package {\em{homotopycontinuation.jl}} \cite{homotopyjl}
Six of them are real
\begin{align*}
& (87.1147, 18.6511, -33.4739, 5.78808, -17.312),  \\
& (80.8137, 30.7661, -23.1126, -3.96875, -28.7833), \\
& (61.9284, -24.7157, -16.0001, 1.76548, 19.994),     \\
& (55.2165, 8.80716, 26.5528, 0.977029, 8.45708),    \\
& (71.9207, -7.85594, -8.51067, 35.9693, 0.649541), \\
& (63.1632, -18.9463, -12.5151, 0.0189543, 24.6219).
\end{align*}
The first line has the lowest Euclidean distance to the estimated point
\begin{equation*}
\hat{\ga}_n=(86.6439,  19.1877, -34.2433,   6.6726, -17.3195),
\end{equation*}
and therefore
\begin{equation*}
\ga_n^{*}=(87.1147, 18.6511, -33.4739, 5.78808, -17.312).
\end{equation*}
Moreover, we have that
\begin{equation*}
\|\ga-\hat{\ga}_n\|=5.2604\quad\text{and}\quad\|\ga-\ga_n^{*}\|=5.0711,
\end{equation*}
so that projecting onto the autocovariance variety improves the empirical estimate.
\end{example}

\begin{figure}[ht]
  \centering
  \includegraphics[scale=0.6]{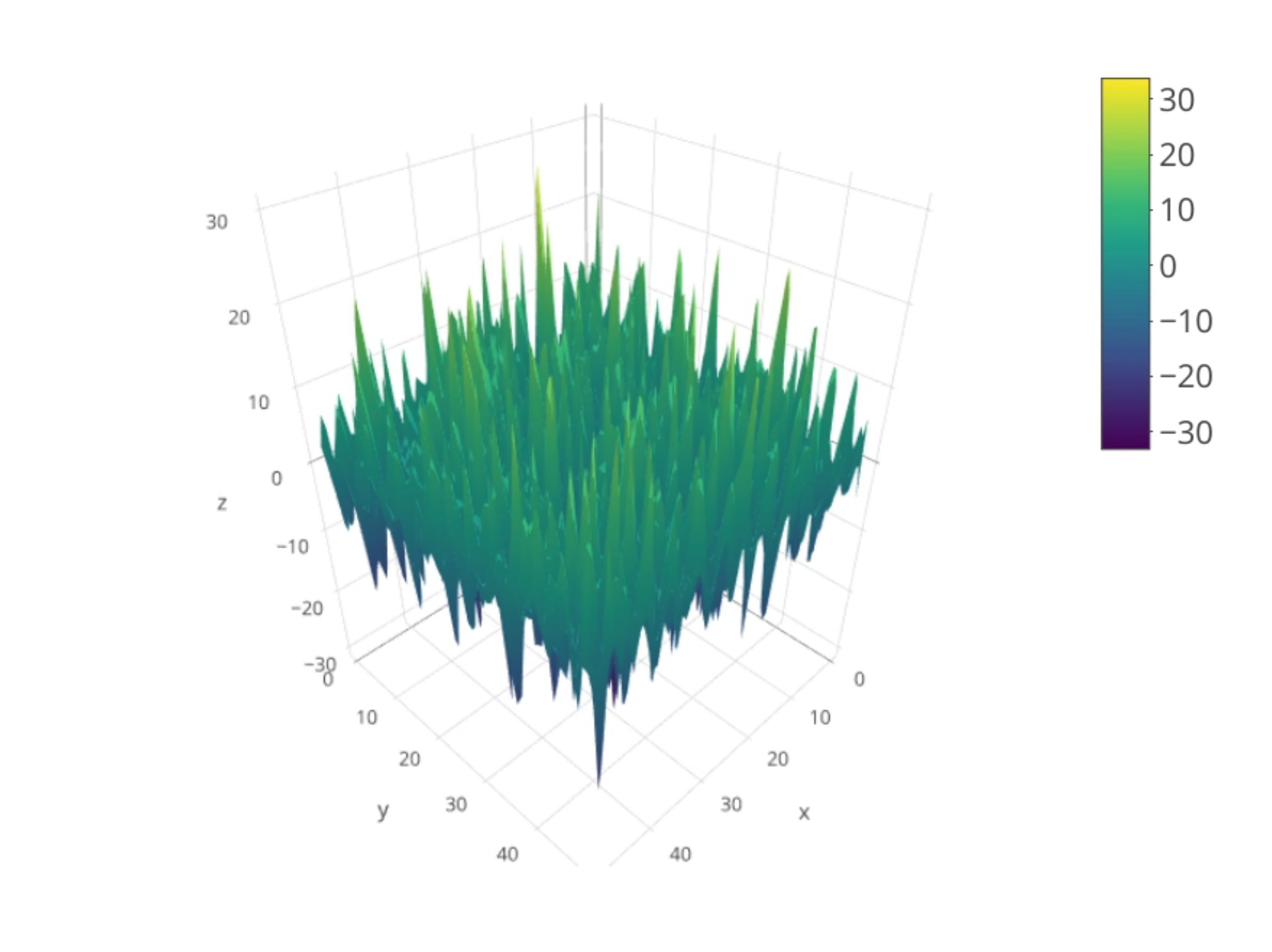}
  \caption{$MA(1,1)$ random field in Example~\ref{ex:MA11}\label{fig1}}
\end{figure}

\begin{table}[H]
\centering
\begin{tabular}{cc}
\hline
$q$ & ED degree \\ 
\hline
(1,1) & 16 \\
(1,1) & 169 \\
(1,2) & 1600 \\
(1,4) &  14641 \\
\hline
\end{tabular}
\caption{ED degree of the variety $\mathcal{MA}_q$ with $q=(1,k)$,  $k=1,2,3,4$}\label{table2}
\end{table}

The computation of the ED degree for $\mathcal{MA}_{(1,k)}$ with $k>1$ is harder than for $\mathcal{MA}_{(1,1)}$. We therefore resort to numerical methods and obtain Table \ref{table2} above. The computations suggest the following pattern.

\begin{conjecture}\label{prop:EDMA11}
The ED degree of $\mathcal{MA}_{(1,k)}$ equals $\frac{(3^{k+1}-1)^2}{4} $ for all $k>0$.
\end{conjecture}

Note that the optimization problem \eqref{eq:lse} gives a point in $\mathcal{MA}_q$ and not a corresponding $a^* \in \PP^Q$.
Theoretically, one could apply the identifiability results of the last section to obtain such $a^*$ by $a^*_n:=\Ga_q^{-1}(\ga_n^{*})$. However, since $\gamma^*_n$ will most often be a numerical approximation, this is not feasible in practice. Instead, one should solve the optimization problem in parametrized form:
%we define the least squares estimator $a^*_n$ as
\begin{equation*}\label{eq:lse2}
a^{*}_n:=\argmin_{a\in\Theta}\|\ga_a-\hat{\ga}_n\|,
\end{equation*}
where $\Theta\subseteq\RR^{Q+1}$ is a compact parameter space. We note then that the ED degree gets multiplied by the algebraic identifiability degree of the model parametrization.

\subsection{Maximum likelihood estimation}

Suppose as before that $(Y_t)_{t\in\bbz^d}$ is a $MA(q)$ random field with mean zero. Furthermore, we assume that $n$ observations $Y(t^1),...,Y(t^n)$ are given, where the vectors $t^1,...,t^n\in\bbz^d$ are ordered according to the lexicographic order. If the driving white noise $(Z_t)_{t\in\bbz^d}$ is Gaussian, then the vector $Y:=(Y(t^1),...,Y(t^n))^\top$ is Gaussian as well, and its likelihood is of the form
\begin{equation*}
L(a)\propto |\Sigma|^{-1/2}\exp\left( -\frac{1}{2}Y^\top\Sigma^{-1}Y \right),
\end{equation*}
where $\Sigma=\Sigma(a)$ is the covariance matrix of $Y$ and $|\Sigma|$ its determinant. The  \emph{maximum likelihood estimator (MLE)} is then defined as the value which maximizes the log-likelihood:
\begin{equation}\label{MLE}
\hat a_n:= \argmax_{a\in\Theta} -\frac{1}{2}\log(|\Sigma|)-\frac{1}{2}Y^\top\Sigma^{-1}Y,
\end{equation}
where $\Theta\subseteq\RR^{Q+1}$ is a compact parameter space.
If $Z$ is not Gaussian, then the latter estimator is called the \emph{quasi maximum likelihood estimator (QMLE)}.

\begin{remark}
In \cite{Yao06} it was shown that under mild assumptions including an invertibility condition, the QMLE $\hat a_n$ is consistent as $n$ tends to infinity. Furthermore, a slightly modified version of $\hat a_n$ (to account for the edge effect) is shown to be asymptotically normal in \cite[Theorem 2]{Yao06}.
\end{remark}

Conveniently, the optimization problem \eqref{MLE} is \emph{still} algebraic, in the sense that the critical or score equations form a system of rational functions of $a\in \Theta$. The number of critical points of the log-likelihood is invariant under generic data $Y$ and this is known as the maximum likelihood degree (\emph{ML degree}).  

We analyze the first nontrivial case, when $q=1$ and $n=2$. Even this simple model is interesting. It has been observed that the MLE can sometimes correspond to non-invertible models, which in this case is equivalent to $|a_0|=|a_1|$, and contrary to what was previously thought, this occurs with positive probability \cite{Cryer81}. 

\begin{proposition}\label{ex:MA1n2}
Consider the $MA(1)$ model with observed sample $Y=(Y_1,Y_2)$. The ML degree is 4, and these four critical points can be divided into three groups:
\begin{enumerate}
\item The parameters $a_0$ and $a_1$ satisfy the two equations
\begin{equation*}
a_0 a_1=Y_1 Y_2\qquad\text{and}\qquad a_0^2+a_1^2=\frac{Y_1^2+Y_2^2}{2}.
\end{equation*}
\item
\begin{equation*}
a_0=a_1=\sqrt{\frac{Y_1^2+Y_2^2-Y_1 Y_2}{3}}
\end{equation*}
\item
\begin{equation*}
a_0=-a_1=\sqrt{\frac{Y_1^2+Y_2^2+Y_1 Y_2}{3}}
\end{equation*}
\end{enumerate}
If $Y_1Y_2=0$, then the MLE corresponds to a degenerate model ($a_0a_1=0)$. Otherwise let $W=\frac{Y_1^2 + Y_2^2}{2Y_1Y_2}$ and the MLE is given as:
\begin{itemize}
\item the point in (3) if $-2<W<0$
\item the point in (2) if $0<W<2$
\item the points in (1) otherwise.
\end{itemize}
\end{proposition}

\begin{proof}
Since we have a $MA(1)$ process and $Y=(Y_1,Y_2)$, we have $\Sigma = \left( \begin{smallmatrix}
\ga(0) & \ga(1)\\
\ga(1)& \ga(0)
\end{smallmatrix} \right)$, and the log-likelihood takes the form 
\begin{equation}
\ell(a_0,a_1) = -\frac{1}{2}\log((a_0^2+a_1^2)^2-a_0^2a_1^2)-\frac{1}{2}(Y_1,Y_2) \begin{pmatrix}
a_0^2+a_1^2 & a_0a_1\\
a_0a_1 & a_0^2+a_1^2
\end{pmatrix}^{-1}(Y_1,Y_2)^\top.
\end{equation}

There are generically four solutions to the system $\frac{\partial \ell}{\partial a_0} = \frac{\partial \ell}{\partial a_1} = 0$. This means the ML degree is 4. The critical points can be divided into the three groups (1), (2) and (3) of the statement. In order to find the MLE depending on the values of $Y_1,Y_2$, we evaluate the likelihood function $\ell$ at these 3 groups of points. In fact, substituting $a_0$ and $a_1$ from (1), (2) and (3) into the log-likelihood function, we obtain
\begin{enumerate}[label=(\roman*)]
\item $-\frac{1}{2} \log \left(\left(Y_1^2-Y_2^2\right)^2\right)-1+\log (2)$,
\item $-\frac{1}{2} \log \left(\frac{1}{3} \left(Y_1^2-Y_1 Y_2+Y_2^2\right)^2\right)-1$,
\item $-\frac{1}{2} \log \left(\frac{1}{3} \left(Y_1^2+Y_1 Y_2+Y_2^2\right)^2\right)-1$.
\end{enumerate}
Computing (i) - (iii) gives the expression  $$\frac{1}{2} \left(\log \left(4 \left(Y_1^2-Y_1
   Y_2+Y_2^2\right)^2\right)-\log \left(3
   \left(Y_1^2-Y_2^2\right)^2\right)\right), $$
which is always nonnegative since
$$ 4 \left(Y_1^2-Y_1
   Y_2+Y_2^2\right)^2- 3
   \left(Y_1^2-Y_2^2\right)^2=\left(Y_1^2-4Y_1
   Y_2+Y_2^2\right)^2\geq0.$$ 
Analogously, (i) is greater than or equal to (ii) from
$ \left(Y_1^2+4Y_1
   Y_2+Y_2^2\right)^2\geq0.$
   
Hence, the first value (i) is always larger than or equal to the values (ii) and (iii), independently of $Y_1$ and $Y_2$. We would conclude that the maximizers are always given by (1), but the points may not be real. Indeed, under (1), if $$a_0a_1 \geq 0 \text{ then } a_0^2+a_1^2 \geq 2a_0a_1 \text{ and thus } W=\frac{Y_1^2 + Y_2^2}{2Y_1Y_2} \geq 2$$ while $$a_0a_1 \leq 0 \text{ implies } a_0^2+a_1^2 \leq -2a_0a_1 \text{ and hence } W \leq -2.$$
Direct inspection reveals that the likelihood for (2) is larger than the one for (3) if and only if $W>0$. Note that when $W=-2$ the points (1) and (3) coincide, while $W=2$ means that (1) and (2) coincide.
\end{proof}

Compare our conditions for $W$ with the similar ones found by \cite{Cryer81} in their effort of computing the distribution of the MLE in this $q=1, n=2$ case (note the different parametrization in terms of $\sigma$, $\theta$). Furthermore, it is gratifying to see that our computations provide a simple explanation for the `curious' phenomenon that the MLE can belong to a non-invertible model. Algebraically, the points in (1) always maximize the likelihood, but for the specified region of $Y_1,Y_2$ these points are strictly complex (even though evaluating at the likelihood yields real values!), which means then that (2) or (3) becomes the MLE.

In \cite{Zhang16}, standard numerical optimization routines were used to find the MLE in samples of $MA(q)$ models with $q=1,2,3,4$. The simulations show the MLE can again lie on the non-invertible boundary. 

\begin{example}
Consider a $MA(1)$ process with $n=3$ sample points $Y=(Y_1,Y_2,Y_3)$. The ML degree is now 8. The expressions for the two non-invertible models $|a_0|=|a_1|$ are:
\begin{align*}
a_0 &=a_1 = \sqrt{\frac{3x_1^2+4x_2^2+3x_3^2-4x_1x_2+2x_1x_3-4x_2x_3}{12}}\\
a_0 &=-a_1 = \sqrt{\frac{3x_1^2+4x_2^2+3x_3^2+4x_1x_2+2x_1x_3+4x_2x_3}{12}} \\
\end{align*}
Obtaining closed form expressions for the other 6 critical points is also possible.
\end{example}

For $n>2$, the matrix $\Sigma$ is tridiagonal: it has $\ga(0)$ in the diagonal and $\ga(1)$ in the upper and lower diagonal. Our ML degree computations of $MA(1)$ for $n=2,3,\dots$ reveal the following pattern:

\begin{conjecture}
The ML degree of $MA(1)$ for $n>1$ sample points is equal to $4(n-1)$.
\end{conjecture}

In contrast, the pattern for $MA(2)$ is not as clear. The first values for $n=2,3,\dots$ are recorded in Table \ref{table3}.
\begin{table}[ht]
\centering
\begin{tabular}{cc}
\hline
$n$ & ML degree \\ 
\hline
3 & 29 \\
4 & 69 \\
5 & 129 \\
6 & 205 \\
\hline
\end{tabular}
\caption{ML degrees of the $MA(2)$ model by number of sample points}\label{table3}
\end{table}

Not unusually, Gr\"obner basis computations quickly become prohibitive. However, this does not mean that our algebraic approach is not useful. In applied algebraic geometry, this often means one needs to go into numerical techniques. Indeed, as far as we know, the algebraic nature of the ML problem has not been exploited yet, and a numerical algebraic geometry approach brings both a fresh perspective and efficient computational tools. Knowing the ML degree beforehand helps homotopy continuation and monodromy methods find all solutions to the critical equations and thus guarantee that the MLE will be found. In contrast, classical local search methods may only find a local maximum of the likelihood function. One way to compare these methods is to conduct simulation studies such as the one in the next example.

\begin{example}\label{ex:simstudy}
We simulate 500 independent paths of a $MA(1)$ process with $n=8$ observations for each path. In Figure~\ref{fig2} a sample path for this process is illustrated.
\begin{figure}[ht]
  \centering
  \includegraphics[width=0.8\textwidth]{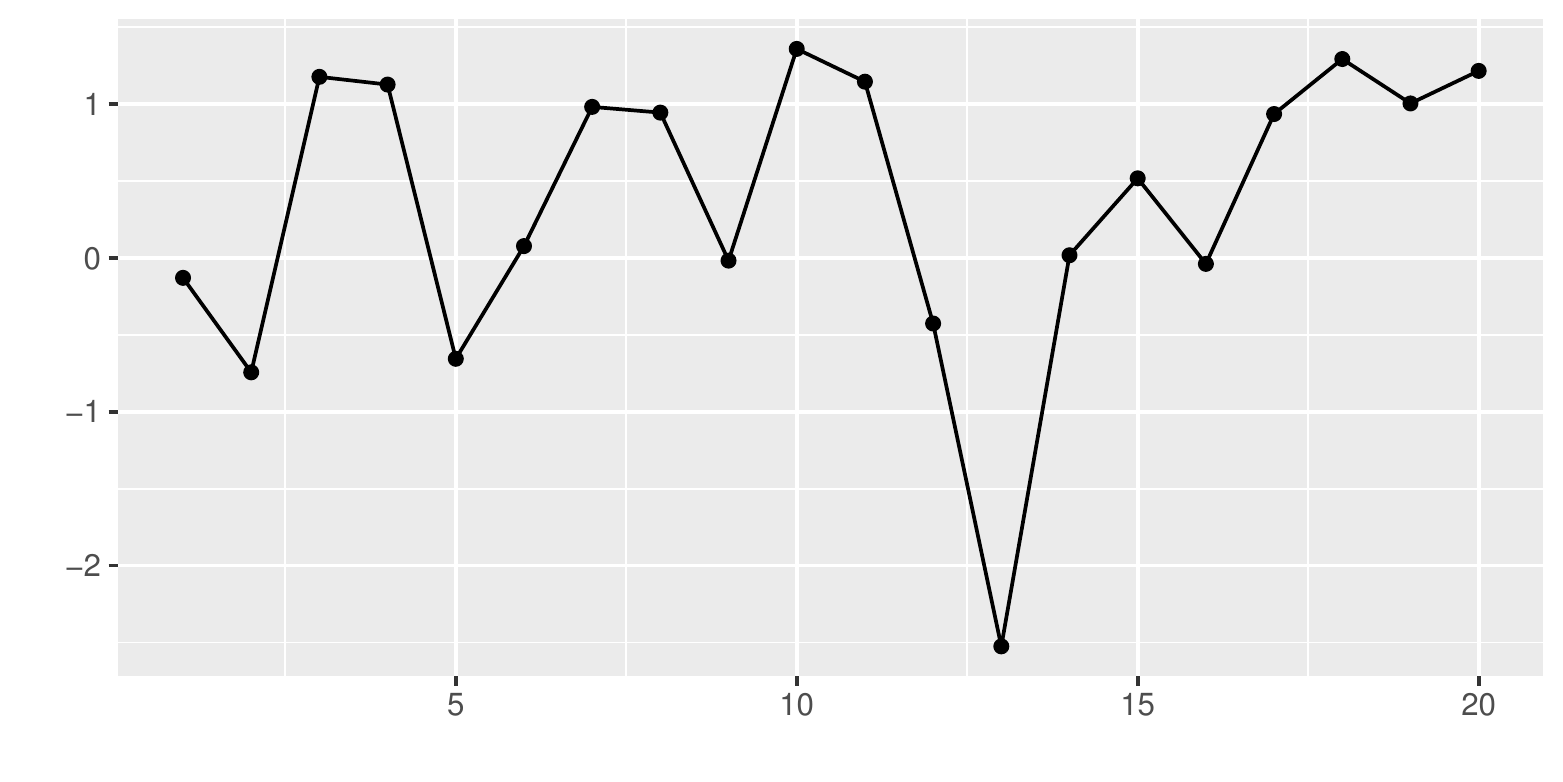}
  \caption{Sample path of a $MA(1)$ process.\label{fig2}}  
\end{figure}
As moving average parameters we take
\begin{equation*}
a_0=1\quad\text{and}\quad a_1=0.5.
\end{equation*}
The $MA(1)$ process is driven by i.i.d. standard Gaussian noise. Having simulated the $MA(1)$ process, we proceed by estimating the model parameters with the MLE in \eqref{MLE} in two different ways. 

For our first approach we use the standard \texttt{R} command \texttt{optim} for minimization of the objective function. As it is standard in time series analysis, we take the output of the innovations algorithm as the initial value for the optimization routine (cf. \cite[Section~2]{Brockwell88}). 

For our second approach we differentiate the likelihood function with respect to the moving average parameters and set the derivatives to zero. In order to compute the critical points of the likelihood function we solve the resulting polynomial system using homotopy continuation. This is implemented in the \texttt{julia} package \textbf{HomotopyContinuation} \cite{homotopyjl}. 

Finally, we evaluate the likelihood at every critical point and choose the maximal one.  The summary of the estimation results are given in Tables~\ref{summary1} and \ref{summary2} below.
\begin{table}[ht]
\centering
\begin{tabular}{rrrrr}
  \hline
 & True Value & Mean & Bias & Std \\ 
  \hline
$a_0$ & 1.0000 & 0.8642 & -0.1358 & 0.2459 \\ 
  $a_1$ & 0.5000 & 0.4503 & -0.0497 & 0.4071 \\ 
   \hline
\end{tabular}
\caption{Parameter estimation results for $MA(1)$ with $n=8$ and \texttt{R} command \texttt{optim}.\label{summary1}}
\end{table}
\begin{table}[ht]
\centering
\begin{tabular}{rrrrr}
  \hline
 & True Value & Mean & Bias & Std \\ 
  \hline
$a_0$ & 1.0000 & 0.8818 & -0.1182 & 0.2129 \\ 
  $a_1$ & 0.5000 & 0.4678 & -0.0322 & 0.5094 \\ 
   \hline
\end{tabular}
\caption{Parameter estimation results for $MA(1)$ with $n=8$ and homotopy continuation.\label{summary2}}
\end{table}

We observe that using homotopy continuation reduces the bias for both $a_0$ and $a_1$, whereas it increases the standard deviation for $a_1$ and decreases the standard deviation for $a_0$.
\end{example}

\vspace{1cm}

Finally, we close this section by reporting the ML degree of $MA(1,1)$:

\begin{proposition}
Assume that $n=4$ sample points $Y=(Y_{11},Y_{12},Y_{21},Y_{22})$ over the lattice $L = \{ 1,2 \}^2$ of a $MA(1,1)$ random field are given. The autocovariance matrix $\Sigma$ of $Y$ is
\begin{equation*}
\Sigma = \begin{pmatrix}
\ga_{00} &\ga_{01} &\ga_{10} &\ga_{11}\\
\ga_{01} &\ga_{00} &\ga_{1-1} &\ga_{10}\\
\ga_{10} &\ga_{1-1} &\ga_{00} &\ga_{01}\\
\ga_{11} &\ga_{10} &\ga_{01} &\ga_{00}
\end{pmatrix}.
\end{equation*}
The ML degree of the model is 192 over $\PP^3$. 
\end{proposition}

\begin{comment}
The next interesting case is when we have a $3 \times 3$ grid of sample points corresponding to the lattice $L=\{1,2,3\}^2$, that is, $n=9$. If we order the observation vector as
$$Y=(Y_{11},Y_{12},Y_{13},Y_{21},Y_{22},Y_{23},Y_{31},Y_{32},Y_{33})$$
then the covariance matrix takes the form
\begin{equation*}
\Sigma = \begin{pmatrix}
\ga_{00}	&\ga_{01} 	&0		 	&\ga_{10}	&\ga_{11}	&0		 	&0		 	&0			&0			\\
\ga_{01} 	&\ga_{00} 	&\ga_{01} 	&\ga_{1-1}	&\ga_{10}	&\ga_{11} 	&0 			&0			&0			\\
0		 	&\ga_{01} 	&\ga_{00} 	&0			&\ga_{1-1}	&\ga_{10} 	&0 			&0			&0			\\
\ga_{10} 	&\ga_{1-1} 	&0		 	&\ga_{00}	&\ga_{01}	&0		 	&\ga_{10} 	&\ga_{11}	&0			\\
\ga_{11} 	&\ga_{10} 	&\ga_{1-1} 	&\ga_{01}	&\ga_{00}	&\ga_{01} 	&\ga_{1-1} 	&\ga_{10}	&\ga_{11}	\\
0		 	&\ga_{11} 	&\ga_{10} 	&0			&\ga_{01}	&\ga_{00} 	&0		 	&\ga_{1-1}	&\ga_{10}	\\
0		 	&0		 	&0		 	&\ga_{10}	&\ga_{1-1}	&0		 	&\ga_{00} 	&\ga_{01}	&0			\\
0		 	&0		 	&0		 	&\ga_{11}	&\ga_{10}	&\ga_{1-1} 	&\ga_{01} 	&\ga_{00}	&\ga_{01}	\\
0		 	&0		 	&0		 	&0			&\ga_{11}	&\ga_{10} 	&0		 	&\ga_{01}	&\ga_{00}
\end{pmatrix}.
\end{equation*}
\end{comment}

%\section{Conclusion and Outlook}

{\bf Acknowledgments.}
We are grateful to Claudia Kl\"uppelberg for her support during this project. We thank Daniele Agostini for helpful conversations. Carlos Am\'endola was partially supported by the Deutsche Forschungsgemeinschaft (DFG) in
the context of the Emmy Noether junior research group KR 4512/1-1. Viet Son Pham acknowledges support from the graduate program TopMath at the Technical University of Munich and the Studienstiftung des deutschen Volkes.

\bigskip

\addcontentsline{toc}{section}{References}
\bibliographystyle{plain}
\setlength{\itemsep}{-1mm}
\bibliography{bib}

%\begin{comment}
\bigskip \bigskip

\footnotesize
\noindent {\bf Authors' addresses:}

\medskip

\noindent 
Zentrum Mathematik, Technische Universit\"at M\"unchen \\ 
Boltzmannstrasse 3, 85748 Garching (b. M\"unchen), Germany \\
\texttt{carlos.amendola@tum.de }, \texttt{ vietson.pham@tum.de}
%\end{comment}

\end{document}